\documentclass[a4paper]{amsart}
\usepackage{amsfonts,amsthm,hyperref,mathtools,xcolor}

\newtheorem{corollary}{Corollary}
\newtheorem{definition}{Definition}
\newtheorem{lemma}{Lemma}
\newtheorem{theorem}{Theorem}
\newtheorem*{ntheorem}{Theorem}

\DeclareMathOperator{\tr}{tr}
\DeclareMathOperator{\Det}{Det}
\DeclareMathOperator{\Tr}{Tr}

\begin{document}

\title{Partial determinants of Kronecker products}

\author{Yorick Hardy}
\address{
 School of Mathematics,
 University of the Witwatersrand,
 Johannesburg 2050, South Africa
}



\begin{abstract}
 Let $\det_2(A)$ be the block-wise determinant (partial determinant).
 We consider the condition for completing the determinant
 \begin{equation*}
  \det(\det_2(A)) = \det(A),
 \end{equation*}
 and characterize the case when for an arbitrary
 Kronecker product $A$ of matrices over an arbitrary field. Further
 insisting that $\det_2(AB)=\det_2(A)\det_2(B)$, for Kronecker products
 $A$ and $B$, yields a multiplicative monoid of matrices.
 This leads to a determinant-root
 operation $\Det$ which satisfies
 \begin{math}
  \Det(\Det_2(A)) = \Det(A)
 \end{math}
 when $A$ is a Kronecker product of matrices for which $\Det$ is defined.
\end{abstract}

\maketitle

\section{Introduction}

Let $\{\,E_{ij}^{[n]}\,:\,i,j\in\{1,\ldots,n\}\,\}$ denote the standard basis
in $M_n(\mathbb{F})$, the algebra of $n\times n$ matrices over a field $\mathbb{F}$.
In other words, $\smash{E_{ij}^{[n]}}$ is an $n\times n$ matrix with a 1 in row $i$ and
column $j$ and all other entries are zero. The standard basis is given
entry-wise by $(\smash{E_{ij}^{[n]}})_{kl}=\delta_{ik}\delta_{jl}$, where $\delta_{ik}$
is the Kronecker delta
\begin{equation*}
 \delta_{ik} = \begin{cases} 1 & i=k, \\ 0 & i\neq k. \end{cases}
\end{equation*}
Let $\otimes$ denote the Kronecker product of matrices, i.e. for $A\in M_m(\mathbb{F})$
and $B\in M_n(\mathbb{F})$ the matrix $A\otimes B$ is given block-wise by
\begin{equation*}
 A\otimes B := \begin{pmatrix}
                (A)_{11}B & (A)_{12}B & \cdots & (A)_{1m} B \\
                (A)_{21}B & (A)_{22}B & \cdots & (A)_{2m} B \\
                 \vdots   & \vdots    & \ddots & \vdots   \\
                (A)_{m1}B & (A)_{m2}B & \cdots & (A)_{mm} B 
               \end{pmatrix}.
\end{equation*}

Choi gives the definition of the partial determinants $\det_1(A)$ and $\det_2(B)$ as follows \cite{choi2017}.

\begin{definition}
 Let $A,B\in M_{mn}(\mathbb{F}) = M_m(M_n(\mathbb{F})) = M_m(\mathbb{F})\otimes M_n(\mathbb{F})$
 with
 \begin{equation*}
  A = \sum_{i,j=1}^n A_{ij}\otimes E_{ij}^{[n]}, \qquad
  B = \sum_{i,j=1}^m E_{ij}^{[m]}\otimes B_{ij}.
 \end{equation*}
 The partial determinants $\det_1(A)$ and $\det_2(B)$ are given by
 \begin{equation*}
  \det_1(A) := \sum_{i,j=1}^n \det(A_{ij}) E_{ij}^{[n]}, \qquad
  \det_2(B) := \sum_{i,j=1}^m \det(B_{ij}) E_{ij}^{[m]}.
 \end{equation*}
\end{definition}

This definition is analogous to that of the partial trace.
The partial trace features prominently in quantum information theory (see for example \cite{carlen}).

\begin{definition}
 Let $A,B\in M_{mn}(\mathbb{F})$ with
 \begin{equation*}
  A = \sum_{i,j=1}^n A_{ij}\otimes E_{ij}^{[n]}, \qquad
  B = \sum_{i,j=1}^m E_{ij}^{[m]}\otimes B_{ij}.
 \end{equation*}
 We define the partial traces
 \begin{equation*}
  \tr_1(A) := \sum_{i,j=1}^n \tr(A_{ij}) E_{ij}^{[n]}, \qquad
  \tr_2(B) := \sum_{i,j=1}^m \tr(B_{ij}) E_{ij}^{[m]}.
 \end{equation*}
\end{definition}

The partial trace is linear while the partial determinant is not linear.
Additionally, the partial trace is ``partial'' in the sense that the trace can be completed
\begin{equation*}
 \tr(\tr_2(A))=\tr(\tr_1(A))=\tr(A).
\end{equation*}
In general, $\det(\det_2(A))\neq\det(A)$.
Thompson showed that if $B\in M_{mn}(\mathbb{C})$ is positive definite, then 
$\det(\det_2(B)) \geq\det(B)$ and that $\det(\det_2(B)) = \det(B)$
if and only if $B_{ij}=\delta_{ij}B_{ij}$, where $\delta_{ij}$ denotes the Kronecker delta \cite{thompson1961}.
This article provides initial results on when equality holds in a more general setting.

Another important property to consider, is whether the identity $\det(AB)=\det(A)\det(B)$
carries over to the partial determinant. In other words, what are the conditions on $A$
and $B$ such that $\det_1(AB)=\det_1(A)\det_1(B)$ ?

Since $B\otimes C$ is permutation similar to $C\otimes B$ via the vec-permutation matrix $P$
(perfect shuffle matrix) \cite{henderson81a,vanloan00a}, so that $\det_2(A)=\det_1(PAP^T)$,
we confine our attention to $\det_1(A)$.

In the following, we only consider partial determinants of Kronecker products.
Let $\circ$ denote the Hadamard product and $A^{(m)}$ denote the
$m$-th Hadamard power of the matrix $A$ (i.e. the entry-wise power).
For every $(0,1)$-matrix $A$ we have $A^{(m)}=A$, and most results
are straightforward for $(0,1)$-matrices and over the field $\mathbb{F}=GF(2)$.
Our main results are characterizations in terms of the distributivity of
the Hadamard power over matrix products. In the case of
$\det_1(AB)=\det_1(A)\det_1(B)$
for Kronecker products, we have the following theorem.

\begin{ntheorem}
 Let $A,B\in M_m(\mathbb{F})$ and $C,D\in M_n(\mathbb{F})$. Then
 \begin{equation*}
  \det_1((AB)\otimes (CD))=\det_1(A\otimes C)\det_1(B\otimes D)
 \end{equation*}
 if and only if $\det(AB)=0$ or $(CD)^{(m)}=C^{(m)}D^{(m)}.$
\end{ntheorem}

The partial determinant may be completed when the $m$-th Hadamard power
and determinant commute, 
$\det(B^{(m)})=(\det(B))^{(m)}=\det(B)^m$.

\begin{ntheorem}
 Let $A\in M_{m}(\mathbb{F})$ and $B\in M_n(\mathbb{F})$. Then
 \begin{equation*}
  \det(\det_1(A\otimes B)) = \det(A\otimes B)
 \end{equation*}
 if and only if $\det(A)=0$ or $\det(B^{(m)}) = \det(B)^m.$
\end{ntheorem}

These two conditions on the matrices in $M_n(\mathbb{F})$ which satisfy the above two
theorems allow us to characterize the matrices in terms of underlying monoids in
$M_n(\mathbb{F})$. As we will see below, the $m$-th Hadamard power distributes over matrix
products and distributes over column/row-wise Grassman products (wedge products)
in these monoids. Finally, we consider a monoid $D_m(\mathbb{F})$ and a determinant-root
operation $\Det$, such that the partial determinant-root obeys

\begin{ntheorem}
 \begin{enumerate}
  \item
   Let $A\in D_m(\mathbb{F})$ and $C\in D_n(\mathbb{F})$. Then
   \begin{equation*}
    \Det(A\otimes C) = \Det(\Det_1(A\otimes C)).
   \end{equation*}
  \item
   Let $A,B\in D_m(\mathbb{F})$ and $C,D\in M_n(\mathbb{F})$. Then
   \begin{equation*}
    \Det_1((A\otimes C)(B\otimes D)) = \Det_1(A\otimes C)\Det_1(B\otimes D).
   \end{equation*}
 \end{enumerate}
\end{ntheorem}

\section{Kronecker products}

Let $I_n$ denote the $n\times n$ identity matrix, and $0_n$ the $n\times n$
zero matrix.
First we note that $\det_1(I_{mn})=\det_1(I_m\otimes I_n)=I_n$ analogous
to $\det(I)=1$.

\begin{lemma}
 \label{lem:hp}%
 Let $A\in M_{m}(\mathbb{F})$ and $B\in M_n(\mathbb{F})$. Then
 \begin{equation*}
  \det_1(A\otimes B) = \det(A)B^{(m)}.
 \end{equation*}
\end{lemma}

\begin{proof}
 The proof follows immediately from
 \begin{equation*}
  A\otimes B = \sum_{i,j=1}^n ((B)_{ij}A)\otimes E_{ij}^{[n]},
 \end{equation*}
 where $B=(B)_{ij}$.
\end{proof}

This is a direct consequence of the remark by Choi that $\det_1(A)$ may be
computed as a determinant where the blocks are treated as scalars and instead
of the usual scalar product we use the Hadamard product \cite{choi2017}.

\begin{theorem}
 \label{thm:sum}%
 Let $A,C\in M_{m}(\mathbb{F})$ and $B,D\in M_n(\mathbb{F})$ with $B\circ D=0_n$. Then
 \begin{equation*}
  \det_1(A\otimes B+C\otimes D) = \det_1(A\otimes B)+\det_1(C\otimes D).
 \end{equation*}
\end{theorem}

\begin{proof}
 We have
 \begin{equation}
  \label{eq:disjsum}
  A\otimes B + C\otimes D
    = \sum_{\textstyle\genfrac{}{}{0pt}{}{i,j=1}{(B)_{ij}\neq0}}^n
        ((B)_{ij}A)\otimes E_{ij}^{[n]}
    + \sum_{\textstyle\genfrac{}{}{0pt}{}{i,j=1}{(D)_{ij}\neq0}}^n
        ((D)_{ij}C)\otimes E_{ij}^{[n]},
 \end{equation}
 and since $B\circ D=0_n$ if follows that $(B)_{ij}\neq0\Longrightarrow (D)_{ij}=0$
 and vice-versa. The two sums in \eqref{eq:disjsum} are over disjoint subsets of
 $\{1,\ldots,n\}^2$ and
 \begin{align*}
  \det_1(A\otimes B+C\otimes D)
    &= \sum_{\textstyle\genfrac{}{}{0pt}{}{i,j=1}{(B)_{ij}\neq0}}^n
        \det(A)(B)_{ij}^m E_{ij}^{[n]}
     + \sum_{\textstyle\genfrac{}{}{0pt}{}{i,j=1}{(D)_{ij}\neq0}}^n
        \det(C)(D)_{ij}^m E_{ij}^{[n]} \\
    &= \det_1(A\otimes B)+\det_1(C\otimes D).
 \end{align*}
\end{proof}

Lemma \ref{lem:hp} provides an immediate characterization of partial
determinants that can be completed.
\begin{theorem}
 \label{thm:parthp}%
 Let $A\in M_{m}(\mathbb{F})$ and $B\in M_n(\mathbb{F})$. Then
 \begin{equation*}
  \det(\det_1(A\otimes B)) = \det(A\otimes B)
 \end{equation*}
 if and only if $\det(A)=0$ or $\det(B^{(m)}) = \det(B)^m.$
\end{theorem}

Drnov\v sek considered a much stronger condition in \cite{drnovsek2015}
for matrices $B$ (which also satisfy Theorem \ref{thm:parthp}), namely
that $B^{(r)}=B^{r}$ for all $r\in\mathbb{N}$. Under the determinant,
we have a weaker condition which leads to somewhat trivial cases for
triangular matrices and $(0,1)$-matrices. For triangular matrices $B$
we have $\det(B^{(m)})=\det(B^m)=\det(B)^m.$

\begin{corollary}
 \label{cor:parttri}%
 Let $A\in M_{m}(\mathbb{F})$ be arbitrary and let $B\in M_n(\mathbb{F})$ be triangular. Then
 \begin{equation*}
  \det(\det_1(A\otimes B)) = \det(A\otimes B).
 \end{equation*}
\end{corollary}

For $(0,1)$-matrices $B$ we have $B^{(m)}=B$, and hence the following corollary
to Theorem \ref{thm:parthp}.

\begin{corollary}
 \label{cor:part01}%
 Let $A\in M_{m}(\mathbb{F})$ be arbitrary and let $B\in M_n(\mathbb{F})$ be a $(0,1)$-matrix. Then
 \begin{equation*}
  \det(\det_1(A\otimes B)) = \det(A\otimes B)
 \end{equation*}
 if and only if $\det(A)\det(B)=0$ or $\det(B)$ is an $(m-1)$-th root of unity.
\end{corollary}

When $n=2$, the characterization is straightforward since
\begin{equation*}
 \det\begin{pmatrix}a^2&b^2\\ c^2&d^2\end{pmatrix} = \det\begin{pmatrix}a&b\\ c&d\end{pmatrix}^2
\end{equation*}
if and only if $b^2c^2=abcd$, so that $bc=0$ (triangular matrix)
or $ad=bc\neq 0$ (rank-1 matrix) and both cases satisfy $\det(B^{(m)}) = \det(B)^m.$
Here we used the fact that a rank-1 matrix $\mathbf{x}\mathbf{y}^T$ remains a
rank-1 matrix $\mathbf{x}^{(m)}(\mathbf{y}^{(m)})^T$
under the Hadamard power, where $\mathbf{x}$ and $\mathbf{y}$ are column vectors.
Now we consider transformations which preserve the partial determinant.
The following theorem characterizes the multiplicative property of the partial determinant.

\begin{theorem}
 \label{thm:mul}%
 Let $A,B\in M_m(\mathbb{F})$ and $C,D\in M_n(\mathbb{F})$. Then
 \begin{equation*}
  \det_1((A\otimes C)(B\otimes D))=\det_1(A\otimes C)\det_1(B\otimes D)
 \end{equation*}
 if and only if $\det(AB)=0$ or $(CD)^{(m)}=C^{(m)}D^{(m)}.$
\end{theorem}

\begin{proof}
 Applying Lemma \ref{lem:hp} on both sides of the equation provides
 \begin{equation*}
  \det(AB)(CD)^{(m)}=\det(AB)C^{(m)}D^{(m)}
 \end{equation*}
 and the result follows.
\end{proof}

When $\mathbb{F}=GF(2)$ we have the following corollary.
\begin{corollary}
 \label{cor:mulgf2}%
 Let $A,B\in M_m(GF(2))$ and $C,D\in M_n(GF(2))$. Then
 \begin{equation*}
  \det_1((A\otimes C)\otimes (B\otimes D))=\det_1(A\otimes C)\det_1(B\otimes D).
 \end{equation*}
\end{corollary}

Since a row or column permutation of a matrix commutes with the Hadamard power,
as does multiplying each row by a constant, we have the following corollary
to Theorem \ref{thm:mul}.

\begin{corollary}
 \label{cor:mulperm}%
 Let $A,B\in M_m(\mathbb{F})$ and $C,P\in M_n(\mathbb{F})$, where $P$ is a permutation matrix
 or a diagonal matrix. Then
 \begin{gather*}
  \det_1((AB)\otimes (PC))=\det_1(A\otimes P)\det_1(B\otimes C), \\
  \det_1((AB)\otimes (CP))=\det_1(A\otimes C)\det_1(B\otimes P).
 \end{gather*}
\end{corollary}

Thus if $B\in M_n(\mathbb{F})$ has no more than $n(n+1)/2$ non-zero entries,
and there exist permutations $P,Q\in M_n(\mathbb{F})$ such that $PBQ^T$ is triangular,
then $\det(\det_1(A\otimes B)) = \det(A\otimes B).$ Of course, this is not true in
general. For example
\begin{equation*}
 B = 
 \begin{pmatrix}
  1 & 0 & 1 & 1 \\
  0 & 1 & 0 & 0 \\
  1 & 0 & 1 & 0 \\
  0 & 0 & 0 & 1
 \end{pmatrix}
\end{equation*}
has $7\leq 10$ non-zero entries, but no such $P$ and $Q$ exist. We note that
multiplying by $P$ preserves the number of non-zero entries in each
column and multiplying by $Q^T$ preserves the number of non-zero entries
in each row.

\section{Monoids for completable partial determinants}

\subsection{Characterizations in terms of monoids}

Theorems \ref{thm:parthp} and \ref{thm:mul} provide a closure result for
matrix multiplication, if $C$ and $D$ are matrices satisfying both theorems
then $CD$ satisfies Theorem \ref{thm:parthp}. 
A set of such matrices, with the identity, form a monoid under matrix
multiplication. This follows since $(CD)^{(m)}=C^{(m)}D^{(m)}$
implies
\begin{equation*}
 \det\left((CD)^{(m)}\right)
  = \det\left(C^{(m)}\right)\det\left(D^{(m)}\right)
  = \det(C)^m\det(D)^m
  = \det(CD)^m
\end{equation*}
when $\det(C^{(m)})=\det(C)^m$ and $\det(D^{(m)})=\det(D)^m$.
The identity $C=I_n$ is clearly in this set. The diagonal matrices
in $M_n(\mathbb{F})$ provide a non-trivial example of such a monoid.
Another example is the set of all matrices with at most one non-zero
row together with the identity matrix.

In the sequel, we use the
term monoid to mean a monoid under the usual matrix multiplication.
Theorems \ref{thm:parthp} and \ref{thm:mul} trivially provide
the next corollary.

\begin{corollary}
 \label{cor:moncor}%
 A monoid $M$ of matrices $C,D\in M_n(\mathbb{F})$ satisfying
 \begin{displaymath}
  \det(C^{(m)}) = \det(C)^m,\quad
  \det(D^{(m)}) = \det(D)^m\quad \text{and}\quad
  (CD)^{(m)} = C^{(m)}D^{(m)}
 \end{displaymath}
 satisfies
 \begin{gather*}
  \det(\det_1(A\otimes B)) = \det(A\otimes B), \\
  \det_1((A\otimes C)\otimes(B\otimes D)) = \det_1(A\otimes C)\det_1(B\otimes D)
 \end{gather*}
 for all $A,B\in M_m(\mathbb{F})$ and $C,D\in M$.
\end{corollary}

A necessary condition for a matrix $C\in M$
to be in such a monoid, is that $(C^k)^{(m)}=(C^{(m)})^k$ for all $k\in\mathbb{N}$.
In the case of $M_n(GF(2))$, the largest such monoid is $M=M_n(GF(2))$.
Let $C_k$ denote the $k$-th column of $C$ and $D_{(k)}$ denote the $k$-th
row of $D$.
Consider the matrix
\begin{equation*}
 \Phi(C,D) 
           := (I_n\otimes C)
              \left(\sum_{i,j=1}^n E_{i,j}^{[n]}\otimes E_{i,j}^{[n]}\right)
              (I_n\otimes D)
            = \sum_{i,j=1}^n E_{i,j}^{[n]}\otimes (C_iD_{(j)}).
\end{equation*}
Then $CD=\tr_1(\Phi(C,D))$ is the sum of the matrices $C_kD_{(k)}$ on
the block diagonal of $\Phi(C,D)$. Similarly, $DC=\tr_2(\Phi(C,D))$.
We also have
\begin{math}
 \Phi(C,D)^{(m)} = \Phi\left(C^{(m)},D^{(m)}\right).
\end{math}
Thus $(CD)^{(m)}=C^{(m)}D^{(m)}$ if and only if
\begin{equation*}
 \tr_1\left(\Phi(C,D)^{(m)}\right) = \left(\tr_1(\Phi(C,D))\right)^{(m)},
\end{equation*}
or equivalently, since $(C_kD_{(k)})^{(m)} = C_{k}^{(m)}D_{(k)}^{(m)}$,
\begin{equation}
  \label{eq:psum}%
  \sum_{k=1}^n C_k^{(m)}D_{(k)}^{(m)}
   = \sum_{(k_1,\ldots,k_m)\in\{1,\ldots,n\}^m}
         (C_{k_1}\circ\cdots\circ C_{k_m})
         (D_{(k_1)}\circ\cdots\circ D_{(k_m)}).
\end{equation}
The set
$\{\,(k,\ldots,k)\,\}\subseteq\{1,\ldots,n\}^m$
selects the block diagonal matrices of the matrix $\Phi(C,D)^{(m)}$ in the sum,
and the non-diagonal block entries sum to zero. When $\mathbb{F}=\mathbb{R}$
and $C$ and $D$ are matrices with non-negative entries, then
$(C_{k_1}\circ\cdots\circ C_{k_m})=(D_{(k_1)}\circ\cdots\circ D_{(k_m)})^T=0$
for all non-diagonal block entries, i.e. $C$ and $D$ are in the monoid $\{\,PDQ^T\,\}$ 
of matrices with at most one non-zero entry in each row and column ($P$ and $Q$
are permutation matrices and $D$ is diagonal).

Let us now consider a stronger condition, namely we consider monoids satisfying
\begin{displaymath}
 \det(C^{(m)}) = \det(C)^m,\quad
 \det(D^{(m)}) = \det(D)^m\quad \text{and}\quad
 (CD)^{(m)} = C^{(m)}D^{(m)}
\end{displaymath}
for all $m\in\mathbb{N}$. On an ordered field we need only consider $m=2$,
with an additional condition on matrix multiplication of elements of the monoid.

The symmetric polynomials play a central role in this characterization. Let
$p_m(x_1,\ldots,x_n)=x_1^m+\cdots+x_n^m$ denote the $m$-th power sum, and
\begin{equation*}
 e_m(x_1,\ldots,x_n) = \sum_{j_1< j_2<\cdots<j_m} x_{j_1}\cdots x_{j_m}.
\end{equation*}
denote the $m$-th elementary symmetric sum \cite{macdonald2015}.
The first power sum and elementary symmetric sum coincide,
$e_1(x_1,\ldots,x_n) = p_1(x_1,\ldots,x_n)$.
The Newton-Girard identities relate the power sums and elementary symmetric sums
as follows \cite[p.~28]{macdonald2015}
\begin{equation}
 \label{eq:ng}
 p_m = \det
       \begin{pmatrix}
          e_1 & 1   & 0   & \cdots & 0 \\
         2e_2 & e_1 & 1   & \cdots & 0 \\
         3e_3 & e_2 & e_1 & \cdots & 0 \\
         \vdots & \vdots & \vdots & \ddots & \vdots \\
         me_m & e_{m-1} & e_{m-2} & \cdots & e_1
       \end{pmatrix}.
\end{equation}

We will make use of the following lemma in some of the characterizations
below. The proof is mechanical but straightforward, we include it here
for completeness.

\begin{lemma}
 \label{lem:ng}%
 Let $\mathbb{F}$ be a field with characteristic 0, and $x_1,\ldots, x_n\in\mathbb{F}$.
 The equations
 \begin{equation*}
  p_m(x_1,\ldots,x_n) = (e_1(x_1,\ldots,x_n))^m\qquad\text{for all~}m\in\mathbb{N}
 \end{equation*}
 hold if and only if $x_jx_k = 0$ for $j\neq k$ ($j,k\in\{1,\ldots,n\}$).
\end{lemma}

\begin{proof}
 The equation $p_2(x_1,\ldots,x_n) = (e_1(x_1,\ldots,x_n))^2$ yields
 $e_2(x_1,\ldots,x_n) = 0$ in \eqref{eq:ng}, when $m=2$. Using
 $e_2(x_1,\ldots,x_n) = 0$ in \eqref{eq:ng}, when $m=3$ yields that
 $e_3(x_1,\ldots,x_n) = 0$. Continuing in the same manner for $p_4$,
 $p_5$, \ldots, $p_n$, we find that $e_n=x_1\cdots x_n=0$. It follows
 that there exists an $x_j$ satisfying $x_j=0$. Now
 $e_{n-1}(x_1,\ldots,x_n) = x_1\cdots x_{j-1}x_{j+1}\cdots x_n = 0$
 which yields another $x_k$ with $x_k=0$, where $k\neq j$. Proceeding
 in this way we arrive
 at $e_2(x_1,\ldots,x_n) = x_j'x_k' = 0$ for some $j'\neq k'$, where
 $x_j=0$ for $j\notin\{j',k'\}$.
\end{proof}

\begin{theorem}
 \label{thm:monequiv}%
 Let $\mathbb{F}$ be an ordered field and $M\subseteq M_n(\mathbb{F})$.
 Then the following statements are equivalent.
 \begin{enumerate}
  \item $M$ is a monoid of matrices satisfying
        \begin{math}
         (CD)^{(m)} = C^{(m)}D^{(m)}
        \end{math}
        for all $C,D\in M$ and $m\in\mathbb{N}$;
  \item $M$ is a monoid of matrices satisfying
        \begin{enumerate}
         \item $(CD)^{(2)}=C^{(2)} D^{(2)}$ for all $C,D\in M$,
         \item for all $C,D\in M$ and all $i,j\in\{1,\ldots,n\}$ there exists $k\in\{1,\ldots,n\}$
               such that $(CD)_{ij}=(C)_{ik}(D)_{kj}.$
        \end{enumerate}
 \end{enumerate}
\end{theorem}

\begin{proof}
 Each monoid includes the identity matrix, and the properties are
 all identically satisfied for $C=I_n$ and $D\in M$.
 Now we show that monoids of type 1 are also monoids of type 2.
 Comparing the entries of the matrices in equation \eqref{eq:psum} in row $i$
 and column $j$ we find
 \begin{equation*}
  \sum_{k=1}^n (C)_{ik}^m(D)_{kj}^m
   = \left(\sum_{k=1}^n(C)_{ik}(D)_{kj}\right)^m.
 \end{equation*}
 Applying Lemma \ref{lem:ng} yields for all $i,j,k,k'\in\{\,1,\dots,n\,\}$ with $k\neq k'$,
 \begin{equation}
  \label{eq:cond}%
  (C)_{ik}(D)_{kj}(C)_{ik'}(D)_{k'j}=0.
 \end{equation}
 Next we show that monoids of type 2 are also monoids of type 1.
 Assume $(CD)^{(2)}=C^{(2)}D^{(2)}$ and that $(CD)_{ij}=(C)_{ik}(D)_{kj}$ for some $k\in\{1,\ldots,n\}$.
 Then
 \begin{equation*}
  ((C)_{ik}(D)_{kj})^2 = ((CD)^{(2)})_{ij} = (C^{(2)}D^{(2)})_{ij} = \sum_{k'=1}^n (C)_{ik'}^2(D)_{k'j}^2
 \end{equation*}
 so that
 \begin{math}
  \displaystyle
  \sum_{\genfrac{}{}{0pt}{2}{k'=1}{k'\neq k}}^n (C)_{ik'}^2(D)_{k'j}^2 = 0
 \end{math}
 and, since $\mathbb{F}$ is an ordered field, equation \eqref{eq:cond} follows.
\end{proof}

\begin{corollary}
 Let $\mathbb{F}$ be an ordered field and let $M\subseteq M_n(\mathbb{F})$.
 Then the following statements are equivalent.
 \begin{enumerate}
  \item $M$ is a monoid of matrices closed under the matrix transpose and satisfying
        \begin{math}
         (CD)^{(m)} = C^{(m)}D^{(m)}
        \end{math}
        for all $C,D\in M$ and $m\in\mathbb{N}$;
  \item $M$ is a submonoid of the monoid
        of matrices with at most one non-zero entry in each row and column.
 \end{enumerate}
\end{corollary}

If we insist that Theorem \ref{thm:sum} holds for all $m\in\mathbb{N}$, applying Lemma \ref{lem:ng}
yields the following theorem, which we state without proof.
\begin{theorem}
 \label{thm:sum2}%
 Let $B,C\in M_n(\mathbb{F})$. Then
 \begin{equation*}
  \det_1(A\otimes (B+C)) = \det_1(A\otimes B)+\det_1(A\otimes C)
 \end{equation*}
 holds for all $m\in\mathbb{N}$ and $A\in M_{m}(\mathbb{F})$
 if and only if $B\circ C=0_n$.
\end{theorem}

For the $2\times 2$ matrices, we find that the property $\det(C^{(m)}) = \det(C)^m$ for all $m\in\mathbb{N}$
is satisfied only when $C$ is triangular, which is consistent with the requirement that $(C^2)^{(2)}=(C^{(2)})^{2}$.
Let $D_n(\mathbb{F})$ denote the set of diagonal $n\times n$ matrices over $\mathbb{F}$, $R_n(\mathbb{F})$ the
set of matrices with at most one non-zero row, and $C_n(\mathbb{F})$ the set of matrices with at most one
non-zero column.

\begin{theorem}
 Let $\mathbb{F}$ be an ordered field and $M\subseteq M_2(\mathbb{F})$.
 Then the following statements are equivalent.
 \begin{enumerate}
  \item $M$ is a monoid of matrices satisfying
        \begin{enumerate}
         \item $\det(C^{(m)}) = \det(C)^m$ for all $C\in M$ and $m\in\mathbb{N}$,
         \item $(CD)^{(m)} = C^{(m)}D^{(m)}$ for all $C,D\in M$ and $m\in\mathbb{N}$;
        \end{enumerate}
  \item $M$ is a monoid of matrices with 
        \begin{enumerate}
         \item $M\subseteq D_2(\mathbb{F})\cup C_2(\mathbb{F})$, or
         \item $M\subseteq D_2(\mathbb{F})\cup R_2(\mathbb{F})$.
        \end{enumerate}
 \end{enumerate}
\end{theorem}

\begin{proof}
 First we prove that monoids of type 1 are monoids of type 2.
 Let
 \begin{equation*}
  C = \begin{pmatrix} c_{11} & c_{12} \cr c_{21} & c_{22} \end{pmatrix} \in M.
 \end{equation*}
 As in the proof of Theorem \ref{thm:monequiv} above, the condition
 $(C^2)^{(2)}=C^{(2)}C^{(2)}$ yields
 \begin{equation*}
  (c_{11}c_{11})(c_{12}c_{21}) = 
  (c_{22}c_{22})(c_{21}c_{12}) = 
  (c_{11}c_{12})(c_{12}c_{22}) = 
  (c_{21}c_{11})(c_{22}c_{21}) = 0.
 \end{equation*}
 The zero matrix satisfies all the conditions, so assume $C\neq 0$.
 If either of $c_{11}$ or $c_{22}$ are non-zero, then $c_{12}c_{21}=0$ and $C$ is
 a triangular matrix with $\det(C^{(m)})=\det(C)^m$. Then $C$ has one of the forms
 \begin{equation*}
  \begin{pmatrix} c_{11} & 0 \\ 0 & c_{22} \end{pmatrix}
  ,\qquad
  \begin{pmatrix} c_{11} & c_{12} \\ 0 & 0 \end{pmatrix}
  ,\qquad
  \begin{pmatrix} c_{11} & 0 \\ c_{21} & 0 \end{pmatrix}
  ,\qquad
  \begin{pmatrix} 0 & c_{12} \\ 0 & c_{22} \end{pmatrix}
  ,\qquad
  \begin{pmatrix} 0 & 0 \\ c_{21} & c_{22} \end{pmatrix}.
 \end{equation*}
 If $c_{11}=c_{22}=0$, then
 $\det(C^{(m)}) = -c_{12}^mc_{21}^m = (-1)^m c_{12}^mc_{21}^m = \det(C)^m.$
 When $m=2$ we find again that $C$ must be of the form
 \begin{displaymath}
  \begin{pmatrix} 0 & c_{12} \\ 0 & 0 \end{pmatrix}
  ,\qquad
  \begin{pmatrix} 0 & 0 \\ c_{21} & 0 \end{pmatrix}.
 \end{displaymath}
 Thus every monoid $M$ of type 1 satisfies
 $M\subseteq D_2(\mathbb{F})\cup C_2(\mathbb{F})\cup R_2(\mathbb{F})$.
 Let $\{\,\mathbf{e}_1,\,\mathbf{e}_2\,\}$ be the standard basis
 in $\mathbb{F}^2$.
 Suppose that $\mathbf{x},\mathbf{y}\in\mathbb{F}^2$ and $j,k\in\{1,2\}$.
 Let $A=\mathbf{x}\mathbf{e}_j^T\in M\cap(C_2(\mathbb{F})\setminus R_2(\mathbb{F}))$
 and
 $B=\mathbf{e}_k\mathbf{y}^T\in M\cap(R_2(\mathbb{F})\setminus C_2(\mathbb{F}))$.
 Then
 \begin{equation*}
  AB = \mathbf{x}\mathbf{e}_j^T\mathbf{e}_k\mathbf{y}^T=\delta_{jk}\mathbf{x}\mathbf{y}^T \in M.
 \end{equation*}
 Thus $\mathbf{x}$ or $\mathbf{y}$ has at most one non-zero entry, which
 contradicts $A\notin R_2(\mathbb{F})$ or contradicts $B\notin C_2(\mathbb{F})$.
 On the other hand,
 \begin{equation*}
  BA = \mathbf{e}_k\mathbf{y}^T\mathbf{x}\mathbf{e}_j^T
     = (\mathbf{y}^T\mathbf{x})\mathbf{e}_k\mathbf{e}_j^T \in M
 \end{equation*}
 if and only if $(\mathbf{y}^T\mathbf{x})^m = (\mathbf{y}^T\mathbf{x})^{(m)}$
 for all $m\in\mathbb{N}$,
 and again $\mathbf{e}_k\mathbf{y}^T\in C_2(\mathbb{F})$ or
 $\mathbf{x}\mathbf{e}_j^T\in R_2(\mathbb{F})$ which yields
 a contradiction. Thus $M\cap(C_2(\mathbb{F})\setminus R_2(\mathbb{F}))=\emptyset$
 or $M\cap(R_2(\mathbb{F})\setminus C_2(\mathbb{F}))=\emptyset$.
 It follows that each monoid of type 1 is a monoid of type 2.

 It is straightforward to
 verify that every monoid of type 2 is also of type 1. Each of the three types
 of matrices (diagonal, at most one non-zero row, at most one non-zero column)
 immediately satisfy $\det(C^{(m)}) = \det(C)^m$ for all $m\in\mathbb{N}$.
 If $D$ is diagonal then $(CD)^{(m)}=C^{(m)}D^{(m)}$ and $(DC)^{(m)}=D^{(m)}C^{(m)}$
 for all $C\in M$.
 In $C_2(\mathbb{F})$ and $R_2(\mathbb{F})$, we have
 \begin{align*}
  ((\mathbf{x}\mathbf{e}_j^T)(\mathbf{y}\mathbf{e}_k^T))^{(m)}
   &= (\mathbf{e}_j^T\mathbf{y})^m\mathbf{x}^{(m)}\mathbf{e}_k^T
   = (\mathbf{x}\mathbf{e}_j^T)^{(m)} (\mathbf{y}\mathbf{e}_k^T)^{(m)}, \\
  ((\mathbf{e}_j\mathbf{x}^T)(\mathbf{e}_k\mathbf{y}^T))^{(m)}
   &= (\mathbf{x}^T\mathbf{e}_k)^m\mathbf{e}_j{\mathbf{y}^{(m)}}^T
   = (\mathbf{e}_j\mathbf{x}^T)^{(m)} (\mathbf{e}_k\mathbf{y}^T)^{(m)},
 \end{align*}
 for any column vectors $\mathbf{x},\mathbf{y}\in\mathbb{F}^2$ and $j,k\in\{1,2\}$.
\end{proof}

We note that some of the observations in the previous proof hold in general.
In a monoid $M\subseteq M_n(\mathbb{F})$, over an ordered field $\mathbb{F}$
and satisfying Theorem \ref{thm:monequiv},
we have for each $C\in M$ and $i,j\in\{1,\ldots,n\}$
\begin{gather}
 \label{eq:sub22}%
 \begin{split}
  (C)_{ii}(C)_{ii}(C)_{ij}(C)_{ji} = 
  (C)_{ii}(C)_{ij}(C)_{ij}(C)_{jj} = 0, \\
  (C)_{jj}(C)_{jj}(C)_{ji}(C)_{ij} = 
  (C)_{ji}(C)_{ii}(C)_{jj}(C)_{ji} = 0.
 \end{split}
\end{gather}
This means that every $2\times 2$ submatrix of $C\in M$ with
entries in columns and rows $i$ and $j$ has one of the forms
\begin{displaymath}
 \begin{pmatrix} a & 0 \\ 0 & b \end{pmatrix}
 ,\qquad
 \begin{pmatrix} 0 & a \\ b & 0 \end{pmatrix}
 ,\qquad
 \begin{pmatrix} a & b \\ 0 & 0 \end{pmatrix}
 ,\qquad
 \begin{pmatrix} 0 & 0 \\ a & b \end{pmatrix}
 ,\qquad
 \begin{pmatrix} a & 0 \\ b & 0 \end{pmatrix}
 ,\qquad
 \begin{pmatrix} 0 & a \\ 0 & b \end{pmatrix}.
\end{displaymath}
The set of equations given by \eqref{eq:sub22} is invariant under
any permutation similarity transformation, i.e. interchanging
$i\leftrightarrow k$ and $j\leftrightarrow l$ leaves \eqref{eq:sub22}
invariant (since we consider all $i$ and $j$). Consider $M_3(\mathbb{F})$
over an ordered field $\mathbb{F}$. The set of equations in \eqref{eq:sub22}
concern three $2\times 2$ submatrices, two of which lie on the diagonal
(indicated by squares) and the outermost $2\times 2$ submatrix, in other
words, we consider the corners of each of the squares:

\begin{equation*}
 C = 
 \left(\,
 \renewcommand{\arraystretch}{1.6}
 \begin{tabular}{|ccc|}
  \cline{1-3}
  \multicolumn{1}{|c}{$c_{11}$} & \multicolumn{1}{c|}{$c_{12}$} & $c_{13}$ \\
  \cline{2-3}
  \multicolumn{1}{|c}{$c_{21}$} & \multicolumn{1}{|c|}{$c_{22}$} & \multicolumn{1}{c|}{$c_{23}$} \\
  \cline{1-2}
  $c_{31}$ & \multicolumn{1}{|c}{$c_{32}$} & \multicolumn{1}{c|}{$c_{33}$} \\
  \cline{1-3}
 \end{tabular}
 \,\right)
\end{equation*}

Any reordering $PCP^T$ of the rows and columns of the matrix $C$ preserves the equations \eqref{eq:sub22}
(here, $P$ is a permutation matrix).
We also have that
\begin{equation*}
 \det(PCP^T)^m = \det((PCP^T)^{(m)})
\end{equation*}
if and only if $\det(C)^m = \det(C^{(m)})$, and that
\begin{equation*}
 ((PCP^T)(PDP^T))^{(m)} = (PCP^T)^{(m)}(PDP^T)^{(m)}
\end{equation*}
if and only if $(CD)^{(m)}=C^{(m)}D^{(m)}$.
Consequently, we may reorder the rows and columns (in the same way) so that non-zero diagonal
entries precede the zero diagonal entries. In the interest of brevity we only consider these
forms of matrices while noting that the observations hold also under permutation similarity.
We consider the 4 cases $c_{11}c_{22}c_{33}\neq0$, $c_{11}c_{22}\neq0$ and $c_{33}=0$,
$c_{11}\neq 0$ and $c_{22}=c_{33}=0$; and $c_{11}=c_{22}=c_{33}=0.$ In the first case
($c_{11}c_{22}c_{33}\neq0$) we have a diagonal matrix and $\det(C)^m=\det(C^{(m)})$ for
all $m\in\mathbb{N}$. In the second case ($c_{11}c_{22}\neq0$ and $c_{33}=0$) we have one of the
matrices
\begin{displaymath}
 \begin{pmatrix}
  c_{11} & 0      & c_{13} \\
  0      & c_{22} & c_{23} \\
  0      & 0      & 0
 \end{pmatrix},\quad
 \begin{pmatrix}
  c_{11} & 0      & 0      \\
  0      & c_{22} & c_{23} \\
  c_{31} & 0      & 0
 \end{pmatrix},\quad
 \begin{pmatrix}
  c_{11} & 0      & c_{13} \\
  0      & c_{22} & 0      \\
  0      & c_{32} & 0
 \end{pmatrix},\quad
 \begin{pmatrix}
  c_{11} & 0      & 0      \\
  0      & c_{22} & 0      \\
  c_{31} & c_{32} & 0
 \end{pmatrix},
\end{displaymath}
and for each matrix $\det(C)^m=\det(C^{(m)})=0$. In the third case
($c_{11}\neq 0$ and $c_{22}=c_{33}=0$) we have one of the matrices
\begin{displaymath}
 \begin{pmatrix}
  c_{11} & 0      & c_{13} \\
  0      & 0      & c_{23} \\
  0      & c_{32} & 0
 \end{pmatrix},\qquad
 \begin{pmatrix}
  c_{11} & 0      & 0      \\
  0      & 0      & c_{23} \\
  c_{31} & c_{32} & 0
 \end{pmatrix}
\end{displaymath}
and for each matrix $\det(C)=-c_{11}c_{23}c_{32}$ so that $\det(C)^2=\det(C^{(2)})$
if and only if $c_{23}c_{32}=0$.
In the fourth case ($c_{11}=c_{22}=c_{33}=0$) we have the matrix
\begin{displaymath}
 \begin{pmatrix}
  0      & c_{12} & c_{13} \\
  c_{21} & 0      & c_{23} \\
  c_{31} & c_{32} & 0
 \end{pmatrix}
\end{displaymath}
and within the monoid $M\subseteq M_3(\mathbb{F})$ satisfying Theorem \ref{thm:monequiv}
we have
\begin{equation*}
 \det C^m = (c_{12}c_{23}c_{31} + c_{13}c_{21}c_{32})^m = c_{12}^mc_{23}^mc_{31}^m + c_{13}^mc_{21}^mc_{32}^m
         = \det{C^{(m)}}
\end{equation*}
since
\begin{equation*}
 (c_{12}c_{23}c_{31})(c_{13}c_{21}c_{32}) = (c_{12}c_{21}c_{13}c_{31}) c_{23}c_{32} = 0.
\end{equation*}
The fourth case then yields one of the matrices (up to permutation similarity)
\begin{displaymath}
 \begin{pmatrix}
  0      & c_{12} & c_{13} \\
  c_{21} & 0      & c_{23} \\
  0      & 0      & 0
 \end{pmatrix},\quad
 \begin{pmatrix}
  0      & c_{12} & 0      \\
  c_{21} & 0      & c_{23} \\
  c_{31} & 0      & 0
 \end{pmatrix},\quad
 \begin{pmatrix}
  0      & c_{12} & c_{13} \\
  c_{21} & 0      & 0      \\
  0      & c_{32} & 0
 \end{pmatrix},\quad
 \begin{pmatrix}
  0      & c_{12} & 0      \\
  c_{21} & 0      & 0      \\
  c_{13} & c_{32} & 0
 \end{pmatrix}.
\end{displaymath}

These observations in $M_2(\mathbb{F})$ and $M_3(\mathbb{F})$ may
be applied in $M_4(\mathbb{F})$, so we consider the $2\times 2$ and
$3\times 3$ submatrices in $M_4(\mathbb{F})$:
\begin{equation*}
 C = 
 \left(\,
 \renewcommand{\arraystretch}{1.6}
 \begin{tabular}{cccc}
  \cline{1-3}
  \multicolumn{1}{|c}{$c_{11}$}
   & \multicolumn{1}{c|}{$c_{12}$}
   & \multicolumn{1}{c|}{$c_{13}$}
   & $c_{14}$ \\
  \cline{2-4}
  \multicolumn{1}{|c}{$c_{21}$}
   & \multicolumn{1}{|c|}{$c_{22}$}
   & \multicolumn{1}{c|}{$c_{23}$}
   & \multicolumn{1}{c|}{$c_{24}$} \\
  \cline{1-4}
  \multicolumn{1}{|c}{$c_{31}$}
   & \multicolumn{1}{|c}{$c_{32}$}
   & \multicolumn{1}{|c|}{$c_{33}$}
   & \multicolumn{1}{c|}{$c_{34}$} \\
  \cline{1-3}
  $c_{41}$
   & \multicolumn{1}{|c}{$c_{42}$}
   & \multicolumn{1}{|c}{$c_{43}$}
   & \multicolumn{1}{c|}{$c_{44}$} \\
  \cline{2-4}
 \end{tabular}
 \,\right)
 =
 \left(\,
 \renewcommand{\arraystretch}{1.6}
 \begin{tabular}{|cccc|}
  \cline{1-4}
  \multicolumn{1}{|c}{$c_{11}$}
   & $c_{12}$
   & \multicolumn{1}{c|}{$c_{13}$}
   & $c_{14}$ \\
  \cline{2-4}
  \multicolumn{1}{|c}{$c_{21}$}
   & \multicolumn{1}{|c}{$c_{22}$}
   & \multicolumn{1}{c|}{$c_{23}$}
   & \multicolumn{1}{c|}{$c_{24}$} \\
  \multicolumn{1}{|c}{$c_{31}$}
   & \multicolumn{1}{|c}{$c_{32}$}
   & \multicolumn{1}{c|}{$c_{33}$}
   & \multicolumn{1}{c|}{$c_{34}$} \\
  \cline{1-3}
  $c_{41}$
   & \multicolumn{1}{|c}{$c_{42}$}
   & $c_{43}$
   & \multicolumn{1}{c|}{$c_{44}$} \\
  \cline{1-4}
 \end{tabular}
 \,\right).
\end{equation*}
In this way, we can proceed to characterize monoids of Corollary \ref{cor:moncor}
in $M_{n+1}(\mathbb{F})$ based on observations in monoids in $M_{n}(\mathbb{F})$,
when $\mathbb{F}$ is an ordered field.

The fact that $\det(C^{(m)})=\det(C)^m$ can be expressed in terms
of the Grassmann product (see for example \cite[p.~173]{merris1997}),
\begin{equation*}
 C_1^{(m)}\wedge\cdots\wedge C_n^{(m)} = (C_1\wedge\cdots\wedge C_n)^{(m)}
\end{equation*}
where
\begin{equation*}
 C_1\wedge\cdots\wedge C_n = \det(C)\mathbf{e}_1\wedge\cdots\wedge\mathbf{e}_n
\end{equation*}
and $\{\,\mathbf{e}_1,\,\ldots,\,\mathbf{e}_n\,\}$ is the standard basis
in $\mathbb{F}^n$. Similarly, $\det(D^{(m)})=\det(D)^m$ is equivalent to
\begin{equation*}
 D_{(1)}^{(m)}\wedge\cdots\wedge D_{(n)}^{(m)} = (D_{(1)}\wedge\cdots\wedge D_{(n)})^{(m)}.
\end{equation*}
To summarize, the Hadamard power distributes over matrix products and
over column/row-wise Grassman products in the monoid of Corollary \ref{cor:moncor}.

\subsection{Determinant-roots}

\newcommand{\Fmul}{\mathbb{F}_{\mathord{\times}}}

The preceding section suggests that a monoid with appropriate
properties may yield a partial determinant that completes in a straightforward way.
In this section we define such monoids and a ``determinant-root'' operation
$\Det$ which is completable on Kronecker products of matrices in these monoids.
First, let us define the $m$-th roots on the multiplicative group $\Fmul$ of $\mathbb{F}$.

Let $G$ be a multiplicative abelian group, and $R_m$ be the subgroup 
\begin{equation*}
 R_m=\{\,g\in G\,:\,g^m=1\,\}
\end{equation*}
of $G$ (which generates the equivalence classes $a\cdot R_m$ where
$a\cdot R_m=b\cdot R_m$ if and only if $a^m=b^m$ for $a,b\in G$). We will
	write $a\cdot R_m=\sqrt[m]{b}$ when $a^m=b$. If no such $a$
exists for a given $b\in G$, then $\sqrt[m]{b}$ is undefined. This definition
of the $m$-th root may be extended with $\sqrt[m]{0}:=0\cdot R_m$ when $G=\Fmul$ is the
multiplicative group of $\mathbb{F}$ and $0$ is the additive identity in $\mathbb{F}$. We note
that $\sqrt[m]{ab}=\sqrt[m]{\vphantom{b}a}\cdot\sqrt[m]{b}$ when the relevant
$m$-th roots exist. The non-zero $m$-th roots of $G$ form a multiplicative group, and
the matrices $M\in M_n(\mathbb{F})$ where $\sqrt[n]{\det(M)}$ exists
form a monoid $D_n(\mathbb{F})$.

\begin{definition}
 We define the determinant-root $\Det$ by
 \begin{displaymath}
  \Det(M):=\sqrt[n]{\det(M)}
 \end{displaymath}
 for all $M\in D_n(\mathbb{F})$.
\end{definition}

This definition provides the usual multiplicative property
\begin{equation*}
 \Det(AB)=\Det(A)\Det(B)
\end{equation*}
for $A,B\in D_n(\mathbb{F})$.
There is a natural embedding $h_{m,n}:a\cdot R_m\mapsto a^n\cdot R_{mn}$
and we may define the product $\ast$ as follows
\begin{equation*}
 (a\cdot R_m)\ast (b\cdot R_n) := h_{m,n}(a\cdot R_m)h_{n,m}(b\cdot R_n) = (a^nb^m)R_{mn}.
\end{equation*}
Consequently, $\sqrt[m]{\vphantom{b}a}\ast\sqrt[n]{b} = \sqrt[mn]{a^nb^m}$ when the given roots exist.
It follows that
\begin{equation*}
 \Det(A\otimes C)=\Det(A)\ast\Det(C)
\end{equation*}
for $A\in D_m(\mathbb{F})$ and $C\in D_n(\mathbb{F})$, from $\det(A\otimes C)=(\det(A))^n(\det(C))^m$.
Let $\sqrt[m]{\vphantom{b}a}$ be the $m$-th root of $a\in\mathbb{F}$. Since $\sqrt[m]{\vphantom{b}a}$
is in the (multiplicative) quotient $\Fmul/R_m$, we consider
the $n$-th root $\sqrt[n]{\sqrt[m]{a}}$, when it exists. In this case, we may write
$\sqrt[mn]{a}\equiv\sqrt[n]{\sqrt[m]{a}}$, since
\begin{equation}
 \label{eq:cong}
 G/R_{mn}(G)\cong (G/R_m(G))\,/\,R_n(G/R_m(G))
\end{equation}
where $R_m(G)$ is the subgroup of an abelian group $G$ as above. The isomorphism is simply
\begin{equation*}
 a R_{mn}(G) \mapsto (a\cdot R_m(G))\cdot R_n(G/R_m(G)).
\end{equation*}



Now we are ready to consider partial determinant-roots and their properties.
The partial determinant-root $\Det_2$ is again the block-wise determinant-root.

\begin{definition}
 Let $A,B\in M_{mn}(\mathbb{F}) = M_m(M_n(\mathbb{F})) = M_m(\mathbb{F})\otimes M_n(\mathbb{F})$
 with
 \begin{equation*}
  A = \sum_{i,j=1}^n A_{ij}\otimes E_{ij}^{[n]}, \qquad
  B = \sum_{i,j=1}^m E_{ij}^{[m]}\otimes B_{ij}
 \end{equation*}
 and assume that $A_{ij}\in D_m(\mathbb{F})$ and $B_{ij}\in D_n(\mathbb{F})$ for all $i,j$.
 The partial determinant-roots $\Det_1(A)$ and $\Det_2(B)$ are given by
 \begin{equation*}
  \Det_1(A) := \sum_{i,j=1}^n \Det(A_{ij}) E_{ij}^{[n]}, \qquad
  \Det_2(B) := \sum_{i,j=1}^m \Det(B_{ij}) E_{ij}^{[m]}.
 \end{equation*}
\end{definition}

In general, we work in the group ring $\mathbb{Z}[\Fmul/R_{m}(\Fmul)]$.
However, since we are concerned only with Kronecker products of matrices, we will express
our results in $\Fmul / R_{m}(\Fmul)$ where possible.

\begin{lemma}\label{lem:detroot}%
 The partial determinant-root obeys
 \begin{equation*}
  \Det_1(A\otimes B) = \Det(A)\,B
 \end{equation*}
 where $\Det(A)B$ is the entry-wise
 product of the entries in $B$ with $\Det(A)$.
\end{lemma}

\begin{proof}
 We have
 \begin{equation*}
  A\otimes B = \sum_{i,j=1}^n ((B)_{ij}A)\otimes E_{ij}^{[n]},
 \end{equation*}
 where $B=(B)_{ij}$. It follows that
 \begin{equation*}
  \Det_1(A\otimes B) = \sum_{i,j=1}^n \Det(B_{ij}A_{ij}) E_{ij}^{[n]}
                     = \sum_{i,j=1}^n B_{ij}\Det(A_{ij}) E_{ij}^{[n]}
 \end{equation*}
 since $\sqrt[m]{B_{ij}^m} = B_{ij}\cdot R_m(\Fmul)$.
\end{proof}

\begin{theorem}
 \begin{enumerate}
  \item
   Let $A\in D_m(\mathbb{F})$ and $C\in D_n(\mathbb{F})$. Then
   \begin{equation*}
    \Det(A\otimes C) = \Det(\Det_1(A\otimes C)).
   \end{equation*}
  \item
   Let $A,B\in D_m(\mathbb{F})$ and $C,D\in M_n(\mathbb{F})$. Then
   \begin{equation*}
    \Det_1((A\otimes C)(B\otimes D)) = \Det_1(A\otimes C)\Det_1(B\otimes D).
   \end{equation*}
 \end{enumerate}
\end{theorem}

\begin{proof}
 Since $\Det(A\otimes C) = \Det(A)\ast\Det(C)$ and 
 \begin{align*}
  \Det(\Det_1(A\otimes C))
    &= \sqrt[n]{\det(\Det(A)C)} \\
    &= \sqrt[n]{(\Det(A))^n\det(C)} \\
    &\equiv \sqrt[mn]{(\det(A))^n(\det(C))^m}
     && \text{(by the isomorphism \eqref{eq:cong})} \\
    &= \Det(A)\ast \Det(C)
 \end{align*}
 statement (1) follows. For statement (2), Lemma \ref{lem:detroot} provides
 \begin{equation*}
  \Det_1((A\otimes C)(B\otimes D)) = \Det(AB)\, (CD).
 \end{equation*}
 Similarly,
 \begin{align*}
  \Det_1(A\otimes C)\Det_1(B\otimes D)
   &= (\Det(A)\, C)(\Det(B)\, D) 
     && \text{(over $\mathbb{Z}[\Fmul /R_m(\Fmul)]$)}\\
   &= (\Det(A)\Det(B))\, (CD) \\
   &= \Det(AB)\, (CD).
     && \qedhere
 \end{align*}
\end{proof}

\section{Conclusion}

We have provided some characterizations of matrices satisfying
\begin{gather*}
 \det(\det_1(A\otimes B)) = \det(A\otimes B), \\
 \det_1((A\otimes C)\otimes(B\otimes D)) = \det_1(A\otimes C)\det_1(B\otimes D).
\end{gather*}
This falls under the problem of partial operations which can be completed, i.e.
we have determined conditions under which the partial determinant
of Kronecker products can be completed.

Let $r:\mathbb{N}\to\mathbb{N}$ and let $g_{n}: M_{n}(\mathbb{F})\to M_{r(n)}(\mathbb{F})$
be some operation on matrices. For example, take $r(n)=1$ and $g_n(A)=\tr(A)$ to be the trace
operation, or $r(n)=n$ and $g_n(A)=A^T$ to be the matrix transpose.
Let $f_{m,n}:M_{mn}(\mathbb{F})\to M_{mr(n)}(\mathbb{F})$ be the partial
operation of $g$ defined by
\begin{equation*}
 f_{m,n}\left(\sum_{i,j=1}^n E_{ij}^{[m]}\otimes A_{ij}\right)
  = \sum_{i,j=1}^n E_{ij}^{[m]} \otimes g_{n}(A_{ij}).
\end{equation*}
The partial operation $f_{m,n}$ on a matrix $M\in M_{mn}(\mathbb{F})$ can be completed when
\begin{equation*}
 g_{mn}(M) = g_{mr(n)}(f_{m,n}(M)).
\end{equation*}
This definition of completability yields that the partial trace is the uniquely completable
linear partial operation of the trace, i.e. $f_{m,n}=\tr_2(M)$.
To see this, let $g_n$ be the trace operation, $g'_n:M_n(\mathbb{F})\to\mathbb{F}$ be linear, and
\begin{equation*}
 f_{m,n}'\left(\sum_{i,j=1}^n E_{ij}^{[m]}\otimes A_{ij}\right)
  = \sum_{i,j=1}^n E_{ij}^{[m]} \otimes g_{n}'(A_{ij})
\end{equation*}
with
\begin{equation*}
 g_{mn}(M) = g_{mr(n)}(f_{m,n}'(M)).
\end{equation*}
We must have
\begin{equation*}
 g_{mn}(E_{ij}^{[m]}\otimes E_{kl}^{[n]})
   = \delta_{ij}\delta_{kl}
   = g_{mr(n)}(E_{ij}^{[n]}\otimes g_{n}'(E_{kl}^{[n]}))
   = \delta_{ij}g_{n}'(E_{kl}^{[n]})
\end{equation*}
from which follows that $g_{n}'(E_{kl}^{[n]})=\delta_{kl}$, i.e. $g_n'=g_n=\tr$.

Consider the partial transpose, i.e. $g_n(A)=A^T$.
Let $M\in M_{mn}(\mathbb{F})$ be the matrix
\begin{equation*}
 M = \sum_{i,j=1}^n E_{ij}^{[m]}\otimes M_{ij}.
\end{equation*}
Then
\begin{equation*}
 f_{m,n}(M)
  = \sum_{i,j=1}^n E_{ij}^{[m]} \otimes g_{n}(M_{ij})
  = \sum_{i,j=1}^n E_{ij}^{[m]} \otimes M_{ij}^T
\end{equation*}
and the partial transpose is not completable, since
\begin{equation*}
 M^T = g_{mn}(M) \neq g_{mn}(f_{m,n}(M)) = \sum_{i,j=1}^n (E_{ij}^{[m]})^T \otimes M_{ij},
\end{equation*}
unless $M$ is block-wise symmetric ($M_{ij}^T=M_{ij}$).

Finally, we note the following regarding partial determinants and
their relation to to partial traces and the exponential map.
When $\mathbb{F}=\mathbb{R}$ or $\mathbb{F}=\mathbb{C}$, we have
\begin{equation*}
 \det(\exp(A)) = \exp(\tr(A))
\end{equation*}
for all square matrices $A$. 
W.-H. Steeb posed the following question in a private communication:
for which matrices $A$ is it true that
$\det_1(\exp(A)) = \exp(\tr_1(A))$? In the case of Kronecker sums,
we have a straightforward answer.

\begin{theorem}
 Let $B\in M_m(\mathbb{F})$ and $C\in M_n(\mathbb{C})$
 and let $A$ be the Kronecker sum $A = B\otimes I_n + I_m\otimes C$.
 Then $\det_1(\exp(A)) = \exp(\tr_1(A))$
 if and only if the $m$-th Hadamard and matrix powers of $\exp(C)$
 coincide, i.e. $(\exp(C))^{(m)} = (\exp(C))^m.$
\end{theorem}


The proof follows from $\exp(A)=\exp(B)\otimes\exp(C)$, Lemma \ref{lem:hp}
and the fact that $\det(\exp(B)) = \exp(\tr(B))$. Assume that $n\neq 0$
in $\mathbb{F}$. Define $\Tr:M_n(\mathbb{F})\to\mathbb{F}$ by $\Tr:A\mapsto\frac{\tr(A)}{n}$,
and define the partial operation $\Tr_1$ in the natural way.
For determinant-roots we obtain
\begin{equation*}
 \Det(\exp(A)) = \exp(\Tr(A))\cdot R_n(\Fmul)
\end{equation*}
where $A\in D_n(\mathbb{F}$). Partial determinant-roots obey the
exponential-determinant-trace relation as follows.

\begin{theorem}
 Let $\mathbb{F}$ be the field $\mathbb{F}=\mathbb{R}$ of real numbers
 or the field $\mathbb{F}=\mathbb{C}$ of complex numbers,
 and let $A$ be the Kronecker sum $A = B\otimes I_n + I_m\otimes C$
 where $B\in D_m(\mathbb{F})$. Then
 \begin{equation*}
  \Det_1(\exp(A)) = (\exp(\Tr_1(A))\cdot R_m(\Fmul))\exp(C).
 \end{equation*}
\end{theorem}

\section*{Acknowledgments}

The author is supported by the National Research Foundation (NRF), South Africa.
This work is based on the research supported in part by the National Research
Foundation of South Africa (Grant Numbers: 105968). Any opinions, findings and
conclusions or recommendations expressed is that of the author(s), and the NRF
accepts no liability whatsoever in this regard.

\bibliographystyle{amsplain}
\bibliography{references}

\end{document}